\documentclass[12pt,a4paper]{article}
\usepackage{indentfirst}
\usepackage{amsfonts}
\usepackage{amssymb}
\usepackage{mathrsfs}
\usepackage{amsmath}
\usepackage{amsthm}
\usepackage{enumerate}
\usepackage{cite}
\allowdisplaybreaks
\newtheorem{defn}{Definition}[section]

\newtheorem{lem}[defn]{Lemma}

\newtheorem{re}[defn]{Remark}
\begin{document}
\title{{\bf Triple Derivations and Triple Homomorphisms of Perfect Lie Superalgebras}}
\author{\normalsize \bf Jia Zhou$^1$, Liangyun Chen$^2$, Yao Ma$^2$}
\date{{\small {$^1$ College of Information Technology, Jilin Agriculture
University,\\ Changchun 130118, China}\\{\small {$^2$ School of
Mathematics and Statistics, Northeast Normal
 University,\\
Changchun 130024, China} }}} \maketitle
\date{}

   {\bf\begin{center}{Abstract}\end{center}}

In this paper, we study triple derivations and triple homomorphisms
of perfect Lie superalgebras over a commutative ring $R$. It is
proved that, if the base ring contains $\frac{1}{2}$, $L$ is a
perfect Lie superalgebra with zero center, then every triple
derivation of $L$ is a derivation, and every triple derivation of
the derivation algebra ${\rm Der}(L)$ is an inner derivation. Let
$L,~L^{'}$ be Lie superalgebras over a commutative ring $R$, the
notion of triple homomorphism from $L$ to $L^{'}$ is introduced. We
prove that, under certain assumptions, homomorphisms,
anti-homomorphisms, and sums of homomorphisms and anti-homomorphisms
are all triple homomorphisms.

\renewcommand{\thefootnote}{\fnsymbol{footnote}}
\footnote[0]{Address correspondence to Prof. Liangyun Chen, School of
Mathematics and Statistics, Northeast
Normal University, Changchun 130024, China; E-mail: chenly640@nenu.edu.cn.}
\footnote[0]{ Supported by  NNSF of China (No. 11171055 and
11471090) and  China Postdoctoral Science
Foundation (No. 2015M581989). }

\section{Introduction}
In studies to derivations of associative algebras
 ({\rm\cite{BC}},{\rm\cite{HJ}},{\rm\cite{S}}) appear naturally
different sorts of triple derivations such as associative triple
derivations, Jordan triple derivations and Lie triple derivations.
Lie triple derivations are interesting not only to studies of
associative rings and associative algebras, but also to studies such
as that of Lie groups{\rm\cite{M4}} and operator
algebras({\rm\cite{M3}},{\rm\cite{JW}},{\rm\cite{ZHC}}). Triple
derivation of Lie algebra is apparently a generalization of
derivation, and is an analogy of triple derivation of associative
algebra and of Jordan algebra. It was first introduced independently
{\rm\cite{M4}} by M$\rm{\ddot u}$ller where it was called
prederivation. It can be easily checked that, for any Lie algebra,
every derivation is a Lie triple derivation, but the converse does
not always hold{\rm\cite{Z1}}.

The relations of homomorphisms, anti-homomorphisms, Jordan
homomorphisms, Lie homomorphisms, and Lie triple homomorphisms
became attractive questions. Bresar gave a characterization of Lie
triple isomorphisms associated to certain associative
algebras{\rm\cite{B}}. Jacobson and Rickart gave some conditions
such that every Jordan homomorphism of a ring is either a
homomorphism or an anti-homomorphism{\rm\cite{J}}. Similar problems
arose in the study of operator
algebras({\rm\cite{M1}},{\rm\cite{M2}}). An analogous result was
proved for more general perfect Lie algebras{\rm\cite{Z2}}.

Lie superalgebras are the natural generalization of Lie algebras,
and have important applications both in mathematics and physics.
 Lie superalgebras are also interesting from a purely mathematical point of view. The aim of this article is to generalize some results
in{\rm\cite{Z1}}and{\rm\cite{Z2}}to the triple derivations and
triple homomorphisms of Lie superalgebras.

Throughout the following sections, $L$ always denote a Lie
superalgebra over a commutative ring $R$ with $1$. A Lie
superalgebra $L$ is called perfect if the derived subalgebra $[L,L]
= L$. For a subset $S$ of $L$, denote by ${\rm C}_{L}(S)$ the
centralizer of $S$ in $L$, and the center of $L$ is denoted by ${\rm
Z}(L)$. $L$ is called centerless if ${\rm Z}(L)=0$. ${\rm Der}(L)$
is the derivation algebra of $L$.

Some definitions needed in this paper are as follows.
\begin{defn}{\rm\cite{K}}
$L=L_{\bar{0}}\oplus L_{\bar{1}}$ is a $\bf Z_{2}$-graded algebra
over a commutative ring $R$ with $1$, we call $L$ a Lie superalgebra
if the multiplication $[~,~]$ satisfies the following identities:
\begin{equation*}
\begin{split}
&(1)\, [x,y]=-(-1)^{\left|x\right|\left|y\right|}[y,x];~~~~~~~~~~~~~~~~~~~~~~~~~~~~~(graded~skew-symmetry)\\
&(2)\,
[x,[y,z]]=[[x,y],z]+(-1)^{\left|x\right|\left|y\right|}[y,[x,z]];~~~~~~~~~~~(graded~Jacobi~identity)
\end{split}
\end{equation*}
where $x,y,z\in hg(L),$  $hg(L)$ denotes the set of all $\bf
Z_2$-homogeneous elements of $L$. If $|x|$ occurs in some expression
in this paper, we always regard $x$ as a $\bf Z_2$-homogeneous
element and $|x|$ as the $\bf Z_2$-degree of $x$.
\end{defn}
\begin{defn}
For a subset $S$ of $L$, the enveloping Lie superalgebra of $S$ is the
Lie subalgebra of $L$ generated by $S$. A Lie superalgebra is called
indecomposable if it cannot be written as a direct sum of two
nontrivial ideals.
\end{defn}
\begin{defn}
An endomorphism $D$ of an R-module $L$  is called a triple
derivation of $L$, if $\forall x, y, z\in L$, D satisfies
\begin{equation*}
D([[x,y],z])=[[D(x),y],z]+(-1)^{\left|D\right|\left|x\right|}[[x,D(y)],z]+(-1)^{\left|D\right|(\left|x\right|+\left|y\right|)}[[x,y],D(z)].
\end{equation*}
\end{defn}
Denote by ${\rm TDer}(L)$ the set of all triple derivations of $L$.
It is not difficult to show that ${\rm TDer}(L)$ is a Lie
superalgebra under the usual bracket of endomorphisms of R-module
(See Lemma 2.1 ).
\begin{defn}
Let $L,~L^{'}$ be two Lie superalgebras over $R$.  An even
$R$-linear mapping $f:L\rightarrow L^{'}$ is called:

(i)\, a homomorphism from $L$ to $L^{'}$ if it satisfies
$f([x,y])=[f(x),f(y)],~\forall x,y\in L.$

(ii)\,  an anti-homomorphism if it satisfies
$f([x,y])=(-1)^{\left|x\right|\left|y\right|}[f(y),f(x)]$,$\forall
x,y\in hg(L).$

(iii)\, a triple homomorphism if it satisfies
$f([x,[y,z]])=[f(x),[f(y),f(z)]],~\forall x,y,z\in L.$

\end{defn}
\begin{defn}
Let $L,~L^{'}$ be Lie superalgebras. A mapping $g:L\rightarrow
L^{'}$ is called a direct sum of $g_{1}$ and $g_{2}$, if
$g=g_{1}+g_{2}$ and there exist ideals $I_{1}, I_{2}$ of the
enveloping Lie superalgebra of $g(L)$ such that $I_{1}\cap I_{2}=0$,
and $g_{1}(L)\subseteq I_{1},~g_{2}(L)\subseteq I_{2}$.
\end{defn}

The main results of this article are the following two theorems.

\noindent {\bf Theorem 1.1}\quad {\rm Let $L$ be a Lie superalgebra
over a commutative ring $R$. If $\frac{1}{2}\in R$, $L$ is perfect
and has zero center, then we have that:
\begin{enumerate}[(1)]
\item ${\rm TDer}(L) = {\rm Der}(L);$
\item ${\rm TDer}({\rm Der}(L)) = ad({\rm Der}(L)).$
\end{enumerate}}
\noindent {\bf Theorem 1.2}\quad {\rm Suppose that $R$ is a
commutative ring with $1$, and $2$ is invertible in $R$. Let $L$ and
$L^{'}$ be Lie superalgebras over $R$, $f$ a triple homomorphism
from $L$ to $L^{'}$, and $M$ the enveloping Lie superalgebra of
$f(L)$. Assume the following statements:
\begin{enumerate}[(1)]
\item $L$ is perfect;
\item $M$ is centerless and can be decomposed into a direct sum of indecomposable ideals.

 Then $f$ is either a homomorphism or an anti-homomorphism or a
direct sum of a homomorphism and an anti-homomorphism.
\end{enumerate}}

\section{Triple Derivations of Perfect Lie Superalgebras}

We proceed to prove $\rm Theorem~1.1$ by the following lemmas.

\begin{lem}\label{lem:1}
For any Lie superalgerba $L$, ${\rm TDer}(L)$ is closed under the
usual Lie bracket.
\end{lem}
\begin{proof}
Let $D_{1},D_{2}\in {\rm TDer}(L),x_{1},x_{2}\in hg(L),x_{3}\in L$.
By the definition of triple derivation, we have
\begin{equation*}
\begin{split}
&D_{1}D_{2}([[x_{1},x_{2}],x_{3}])\\
=&D_{1}([[D_{2}(x_{1}),x_{2}],x_{3}]+(-1)^{\left|D_{2}\right|\left|x_{1}\right|}[[x_{1},D_{2}(x_{2}],x_{3}]+(-1)^{\left|D_{2}\right|(\left|x_{1}\right|+\left|x_{2}\right|)}[[x_{1},x_{2}],D_{2}(x_{3})])\\
=&[D_{1}D_{2}(x_{1}),x_{2}],x_{3}]+(-1)^{\left|D_{2}\right|(\left|x_{1}\right|+\left|x_{2}\right|)}(-1)^{\left|D_{1}\right|(\left|x_{1}\right|+\left|x_{2}\right|)}[[x_{1},x_{2}],D_{1}D_{2}(x_{3})]\\
&+(-1)^{\left|D_{2}\right|\left|x_{1}\right|}[[D_{1}(x_{1}),D_{2}(x_{2})],x_{3}]+(-1)^{\left|D_{1}\right|(\left|D_{2}\right|+\left|x_{1}\right|)}[[D_{2}(x_{1}),D_{1}(x_{2})],x_{3}]\\
&+(-1)^{\left|D_{1}\right|(\left|D_{2}\right|+\left|x_{1}\right|+\left|x_{2}\right|)}[[D_{2}(x_{1}),x_{2}],D_{1}(x_{3})]+(-1)^{\left|D_{2}\right|\left|x_{1}\right|}(-1)^{\left|D_{1}\right|\left|x_{1}\right|}[[x_{1},D_{1}D_{2}(x_{2})],x_{3}]\\
&+(-1)^{\left|D_{2}\right|(\left|x_{1}\right|+\left|x_{2}\right|)}(-1)^{\left|D_{1}\right|\left|x_{1}\right|}[[x_{1},D_{1}(x_{2})],D_{2}(x_{3})]+(-1)^{\left|D_{2}\right|(\left|x_{1}\right|+\left|x_{2}\right|)}[[D_{1}(x_{1}),x_{2}],D_{2}(x_{3})]\\
&+(-1)^{\left|D_{2}\right|\left|x_{1}\right|}(-1)^{\left|D_{1}\right|(\left|D_{2}\right|+\left|x_{1}\right|+\left|x_{2}\right|)}[[x_{1},D_{2}(x_{2})],D_{1}(x_{3})],
\end{split}
\end{equation*}
and
\begin{equation*}
\begin{split}
&D_{2}D_{1}([[x_{1},x_{2}],x_{3}])\\
=&D_{2}([[D_{1}(x_{1}),x_{2}],x_{3}]+(-1)^{\left|D_{1}\right|\left|x_{1}\right|}[[x_{1},D_{1}(x_{2}],x_{3}]+(-1)^{\left|D_{1}\right|(\left|x_{1}\right|+\left|x_{2}\right|)}[[x_{1},x_{2}],D_{1}(x_{3})])\\
=&[[D_{2}D_{1}(x_{1}),x_{2}],x_{3}]+(-1)^{\left|D_{1}\right|(\left|x_{1}\right|+\left|x_{2}\right|)}(-1)^{\left|D_{2}\right|(\left|x_{1}\right|+\left|x_{2}\right|)}[[x_{1},x_{2}],D_{2}D_{1}(x_{3})]\\
&+(-1)^{\left|D_{1}\right|\left|x_{1}\right|}[[D_{2}(x_{1}),D_{1}(x_{2})],x_{3}]+(-1)^{\left|D_{2}\right|(\left|D_{1}\right|+\left|x_{1}\right|)}[[D_{1}(x_{1}),D_{2}(x_{2})],x_{3}]\\
&+(-1)^{\left|D_{2}\right|(\left|D_{1}\right|+\left|x_{1}\right|+\left|x_{2}\right|)}[[D_{1}(x_{1}),x_{2}],D_{2}(x_{3})]+(-1)^{\left|D_{1}\right|\left|x_{1}\right|}(-1)^{\left|D_{2}\right|\left|x_{1}\right|}[[x_{1},D_{2}D_{1}(x_{2})],x_{3}]\\
&+(-1)^{\left|D_{1}\right|(\left|x_{1}\right|+\left|x_{2}\right|)}(-1)^{\left|D_{2}\right|\left|x_{1}\right|}[[x_{1},D_{2}(x_{2})],D_{1}(x_{3})]+(-1)^{\left|D_{1}\right|(\left|x_{1}\right|+\left|x_{2}\right|)}[[D_{2}(x_{1}),x_{2}],D_{1}(x_{3})]\\
&+(-1)^{\left|D_{1}\right|\left|x_{1}\right|}(-1)^{\left|D_{2}\right|(\left|D_{1}\right|+\left|x_{1}\right|+\left|x_{2}\right|)}[[x_{1},D_{1}(x_{2})],D_{2}(x_{3})].
\end{split}
\end{equation*}
Then from easy computation we have
\begin{equation*}
\begin{split}
&[D_{1},D_{2}]([[x_{1},x_{2}],x_{3}])\\
=&D_{1}D_{2}([[x_{1},x_{2}],x_{3}])-(-1)^{\left|D_{1}\right|\left|D_{2}\right|}D_{2}D_{1}([[x_{1},x_{2}],x_{3}])\\
=&[[[D_{1},D_{2}](x_{1}),x_{2}],x_{3}]+(-1)^{\left|x_{1}\right|(\left|D_{1}\right|+\left|D_{2}\right|)}[[x_{1},[D_{1},D_{2}](x_{2})],x_{3}]\\
&+(-1)^{(\left|D_{1}\right|+\left|D_{2}\right|)(\left|x_{1}\right|+\left|x_{2}\right|)}[[x_{1},x_{2}],[D_{1},D_{2}](x_{3})].
\end{split}
\end{equation*}
Hence, $[D_{1},D_{2}]\in {\rm TDer}(L).$ The lemma is proved.
\end{proof}

Clearly, ${\rm ad}(L), {\rm Der}(L)$ are all subalgebras of ${\rm
TDer}(L)$. Since $L$ is perfect, every element $x\in hg(L)$ can be
written as a finite sum of Lie brackets, i.e., there exists a finite
index set $I$ such that $x=\sum_{i\in I\atop
\left|x_{i1}\right|+\left|x_{i2}\right|=\left|x\right|}[x_{i1},x_{i2}],$
for some $x_{i1},x_{i2}\in L.$ In this article, we always put $\sum$
in place of $\sum_{i\in I\atop
\left|x_{i1}\right|+\left|x_{i2}\right|=\left|x\right|}$ for
convenience.

Moreover, we have the following lemma.
\begin{lem}\label{lem:1}
If $L$ is perfect, then ${\rm ad}(L)$ is an ideal of Lie
superalgebra ${\rm TDer}(L)$.
\end{lem}
\begin{proof}
Let $D\in {\rm TDer}(L), x\in hg(L)$. $\forall z\in L$, we have
\begin{align*}
 &[D,{\rm ad}x](z)=D{\rm ad}x(z)-(-1)^{\left|D\right|\left|x\right|}{\rm ad}x(D(z))\\
&=D[x,z]-(-1)^{\left|D\right|\left|x\right|}[x,D(z)]\\
 &=D([\sum[x_{i1},x_{i2}],z])-(-1)^{\left|D\right|\left|x\right|}[\sum[x_{i1},x_{i2}],D(z)]\\
 &=\sum D([[x_{i1},x_{i2}],z])-\sum(-1)^{\left|D\right|\left|x\right|}[[x_{i1},x_{i2}],D(z)]\\
 &=\sum[[D(x_{i1}),x_{i2}],z]+\sum(-1)^{\left|D\right|\left|x_{i1}\right|}[[x_{i1},D(x_{i2})],z]\\
 &+\sum(-1)^{\left|D\right|\left|x\right|}[[x_{i1},x_{i2}],D(z)]-\sum(-1)^{\left|D\right|\left|x\right|}[[x_{i1},x_{i2}],D(z)]\\
&={\rm
ad}(\sum[D(x_{i1}),x_{i2}]+\sum(-1)^{\left|D\right|\left|x_{i1}\right|}[x_{i1},D(x_{i2})])(z).
\end{align*}
By the arbitrariness of $z$, $[D,{\rm ad}x]$ is an inner derivation.
Hence, ${\rm ad}(L)$ is an ideal of ${\rm TDer}(L)$. The lemma
holds.
\end{proof}

\begin{lem}\label{lem:1-2}
If $L$ is a perfect Lie superalgebra with zero center, then there
exits an $R$-module homomorphism $\delta: {\rm TDer}(L)\rightarrow
{\rm End}(L), \delta(D)=\delta_{D}$ such that $\forall x\in L,D\in
{\rm TDer}(L)$, one has $[D,{\rm ad}x]={\rm ad}\delta_{D}(x).$
\end{lem}
\begin{proof}
From $\rm Lemma~2.2$, if $L$ is perfect and $L$ has zero center,
$D\in {\rm TDer}(L)$, we can define a module endomorphism
$\delta_{D}$ on $L$ such that for $\forall x=\sum[x_{i1},x_{i2}]\in
hg(L),$
\begin{equation*}
\begin{split}
&\delta_{D}(x)=\sum([D(x_{i1}),x_{i2}]+(-1)^{\left|D\right|\left|x_{i1}\right|}[x_{i1},D(x_{i2})]).
\end{split}
\end{equation*}
In fact, the definition is independent of the form of expression of
$x$. For proving it, let
\begin{equation*}
\begin{split}
&\alpha=\sum([D(x_{i1}),x_{i2}]+(-1)^{\left|D\right|\left|x_{i1}\right|}[x_{i1},D(x_{i2})]).
\end{split}
\end{equation*}
 If there exists another finite index set
$J$ and $y_{ji}\in L$ such that $x$ can be expressed in the form
$x=\sum_{j\in
J\atop\left|y_{j1}\right|+\left|y_{j2}\right|=\left|x\right|}[y_{j1},y_{j2}],$
we also put $\sum$ in place of $\sum_{j\in J\atop
\left|y_{i1}\right|+\left|y_{i2}\right|=\left|x\right|}$ when
necessary, let
\begin{equation*}
\begin{split}
&\beta=\sum([D(y_{j1}),y_{j2}]+(-1)^{\left|D\right|\left|y_{j1}\right|}[y_{j1},D(y_{j2})]).
\end{split}
\end{equation*}
Since $D\in {\rm TDer}(L),\forall z\in L$, we have
\begin{align*}
 &[\alpha,z]=\sum([[D(x_{i1}),x_{i2}],z]+(-1)^{\left|D\right|\left|x_{i1}\right|}[[x_{i1},D(x_{i2})],z])\\
&=\sum(D([[x_{i1},x_{i2}],z])-(-1)^{\left|D\right|\left|x\right|}[[x_{i1},x_{i2}],D(z)])\\
 &=D([x,z])-(-1)^{\left|D\right|\left|x\right|}[x,D(z)]\\
 &=\sum(D([[y_{j1},y_{j2}],z])-(-1)^{\left|D\right|\left|x\right|}[[y_{j1},y_{j2}],D(z)])\\
 &=\sum([[D(y_{j1}),y_{j2}],z]+(-1)^{\left|D\right|\left|y_{j1}\right|}[[y_{j1},D(y_{j2})],z])=[\beta,z].
\end{align*}
Hence, $[\alpha-\beta,z]=0.$ This means that $\alpha-\beta\in {\rm
Z}(L).$ Since ${\rm Z}(L)=0,~\alpha=\beta.$ Hence, $\delta_{D}$ is
well-defined. The rests of the lemma follow from the proof of $\rm
Lemma~2.2$.
\end{proof}

Using the mapping $\delta_{D}$ and the proof of $\rm Lemma~2.2$, we
have the following lemmas.
\begin{lem}\label{lem:1-3}
If $L$ is a perfect Lie superalgebra with zero center, then for
$\forall D\in {\rm TDer}(L), \delta_{D}\in {\rm Der}(L).$
\end{lem}
\begin{proof}
Suppose $D\in {\rm TDer}(L), x\in hg(L),y\in L$. Then $[D,{\rm
ad}([x,y])]={\rm ad}\delta_{D}([x,y]).$

In other hand,
\begin{equation*}
\begin{split}
&[D,{\rm ad}([x,y])]=[D,[{\rm ad}x,{\rm ad}y]]\\
&=[[D,{\rm ad}x],{\rm ad}y]+(-1)^{\left|D\right|\left|x\right|}[{\rm ad}x,[D,{\rm ad}y]]\\
&=[{\rm ad}\delta_{D}(x),{\rm ad}y]+(-1)^{\left|D\right|\left|x\right|}[{\rm ad}x,{\rm ad}\delta_{D}(y)]\\
&={\rm ad}([\delta_{D}(x),y]+(-1)^{\left|D\right|\left|x\right|}[x,\delta_{D}(y)]).\\
\end{split}
\end{equation*}
Hence, ${\rm ad}\delta_{D}([x,y])={\rm
ad}([\delta_{D}(x),y]+(-1)^{\left|D\right|\left|x\right|}[x,\delta_{D}(y)]).$
Since ${\rm
Z}(L)=0,~\delta_{D}([x,y]=[\delta_{D}(x),y]+(-1)^{\left|D\right|\left|x\right|}[x,\delta_{D}(y)]$.
By the arbitrariness of $x,y,\delta_{D}\in {\rm Der}(L).$
\end{proof}

\begin{lem}\label{lem:2}
If the base ring $R$ contains $\frac{1}{2}, L$ is perfect, then the
centralizer of ${\rm ad}(L)$ in ${\rm TDer}(L)$is trivial, i.e.,
${\rm C}_{{\rm TDer}L}({\rm ad}(L))=0$. In particular, the center of
${\rm TDer}(L)$ is zero.
\end{lem}
\begin{proof}
Let $D\in {\rm C}_{{\rm TDer}L}({\rm ad}(L)).$ Then for $\forall
x\in L$, $[D,{\rm ad}x]=0.$ Hence, for $\forall x,y\in hg(L),
D([x,y])-(-1)^{\left|D\right|\left|x\right|}[x,D(y)]=[D,{\rm
ad}x](y)=0.$ Thus,
$D([x,y])=[D(x),y]=(-1)^{\left|D\right|\left|x\right|}[x,D(y)]$.

For $x_{1},x_{2},x_{3}\in hg(L),$ we always have that
\begin{equation*}
\begin{split}
&D([[x_{1},x_{2}],x_{3}])=(-1)^{\left|D\right|(\left|x_{1}\right|+\left|x_{2}\right|)}[[x_{1},x_{2}],D(x_{3})]\\
&=(-1)^{\left|D\right|\left|x_{1}\right|}[[x_{1},D(x_{2})],x_{3}]=[[D(x_{1}),x_{2}],x_{3}].
\end{split}
\end{equation*}
Therefore,
\begin{equation*}
\begin{split}
&D([[x_{1},x_{2}],x_{3}])\\
=&[[D(x_{1}),x_{2}],x_{3}]+(-1)^{\left|D\right|\left|x_{1}\right|}[[x_{1},D(x_{2})],x_{3}]+(-1)^{\left|D\right|(\left|x_{1}\right|+\left|x_{2}\right|)}[[x_{1},x_{2}],D(x_{3})]\\
=&3D([[x_{1},x_{2}],x_{3}]).
\end{split}
\end{equation*}
Hence, $2D([[x_{1},x_{2}],x_{3}])=0.$ Because $\frac{1}{2}\in
R,~D([[x_{1},x_{2}],x_{3}])=0.$ Since $L$ is perfect, every element
of $L$ can be expressed as the linear combination of elements of the
form $[[x_{1},x_{2}],x_{3}]$, we have that $D=0$. This completes the
proof.
\end{proof}

The next lemma is well known.
\begin{lem}\label{lem:2}{\rm\cite{K}}
For all Lie superalgebra $L$, if $x\in L,~D\in {\rm Der}(L)$, then
$[D,{\rm ad}x]={\rm ad}(D(x)).$
\end{lem}

Now we can prove the first conclusion of the theorem.

\begin{lem}\label{lem:2-2}
If the base ring $R$ contains $\frac{1}{2}$, $L$ is perfect and has
trivial center, then ${\rm TDer}(L)={\rm Der}(L).$
\end{lem}
\begin{proof}
Suppose $x\in L,~D\in {\rm TDer}(L)$. By $\rm Lemma~2.3,~[D,{\rm
ad}x]={\rm ad}\delta_{D}(x).$ By $\rm Lemma~2.4$ and $\rm
Lemma~2.6$, ${\rm ad}\delta_{D}(x)=[\delta_{D},{\rm ad}x].$ Hence,
$D-\delta_{D}\in {\rm C}_{{\rm TDer}(L)}({\rm ad}(L)).$ By $\rm
Lemma~2.5$, $D-\delta_{D}=0, i.e., D=\delta_{D}\in {\rm Der}(L)$.
Hence, ${\rm TDer}(L)\subseteq {\rm Der}(L).$ The lemma follows from
$\rm Lemma~2.4$.
\end{proof}
The remainder of the chapter aims to prove the second conclusion of
$\rm theorem~1.1$.
\begin{lem}\label{lem:2-3}
If $L$ is a perfect Lie superalgebra, $D\in {\rm TDer}({\rm
Der}(L))$, then $D({\rm ad}(L))\subseteq {\rm ad}(L).$
\end{lem}
\begin{proof}
Since $L$ is perfect, we have
\begin{align*}
 &D({\rm ad}x)=\sum D({\rm ad}[[x_{i1},x_{i2}],x_{i3}])\\
&=\sum D([[{\rm ad}x_{i1},{\rm ad}x_{i2}],{\rm ad}x_{i3}])\\
 &=\sum([[D({\rm ad}x_{i1}),{\rm ad}x_{i2}],{\rm ad}x_{i3}]+(-1)^{\left|D\right|\left|x_{i1}\right|}[[{\rm ad}x_{i1},D({\rm ad}x_{i2})],{\rm ad}x_{i3}]\\
 &+(-1)^{\left|D\right|(\left|x_{i1}\right|+\left|x_{i2}\right|)}[[{\rm ad}x_{i1},{\rm ad}x_{i2}],D({\rm ad}x_{i3})].
\end{align*}
Hence, $D({\rm ad}x)\in {\rm ad}(L)$. The lemma holds thanks to $\rm
Lemma~2.2$.
\end{proof}

\begin{lem}\label{lem:2-2}
Suppose that $R$ is the base ring containing $\frac{1}{2}$, $L$ is a
perfect Lie superalgebra with zero center, $D\in {\rm TDer}({\rm
Der}(L)).$ If $D({\rm ad}(L))=0,$ then $D=0.$
\end{lem}
\begin{proof}
For $\forall d\in {\rm Der}(L),~x\in hg(L),$ since $L$ is perfect,
$x=\sum[x_{i1},x_{i2}].$ We have that
\begin{align*}
 &[{\rm ad}x,D(d)]=[\sum[{\rm ad}x_{i1},{\rm ad}x_{i2}],D(d)]\\
 &=\sum((-1)^{\left|D\right|\left|x\right|}D([[{\rm ad}x_{i1},{\rm ad}x_{i2}],d])-(-1)^{\left|D\right|\left|x\right|}[[D({\rm ad}x_{i1}),{\rm ad}x_{i2}],d]\\
 &-(-1)^{\left|D\right|\left|x\right|}(-1)^{\left|D\right|\left|x_{i1}\right|}[[{\rm ad}x_{i1},D({\rm ad}x_{i2})],d]).
\end{align*}
By $\rm Lemma~2.2$, $[{\rm ad}x,d]\in {\rm ad}(L)$, so $D([{\rm
ad}x,d])=0$. Hence, $[{\rm ad}x,D(d)]=0.$ Therefore, $D(d)\in {\rm
C}_{{\rm TDer}(L)}({\rm ad}(L)).$ By $\rm Lemma~ 2.5$, $D(d)=0$.
Hence, $D=0$. The lemma holds.
\end{proof}

\begin{lem}\label{lem:2-3}
Let $L$ is a Lie superalgebra over commutative ring $R$. Suppose
that $\frac{1}{2}\in R$, $L$ is perfect and has zero center. If
$D\in {\rm TDer}({\rm Der}(L))$, then there exists $d\in {\rm
Der}(L)$ such that for $\forall x \in L,~D({\rm ad}x)={\rm
ad}(d(x))$.
\end{lem}
\begin{proof}
For $\forall D\in {\rm TDer}({\rm Der}(L)),~x\in L,$ by $\rm
Lemma~2.8,$ $D({\rm ad}x)\in {\rm ad}(L)$. Let $y\in L$ and $D({\rm
ad}x)={\rm ad}y$. Since the center ${\rm Z}(L)$ is trivial, such $y$
is unique. Clearly, the map $d:L\rightarrow L$ given by $d(x)=y$ is
a $R$-module endomorphism of $L$. Let $x_{1},x_{2}\in hg(L),x_{3}\in
L$. We have
\begin{align*}
 &{\rm ad}d([[x_{1},x_{2}],x_{3}])=D({\rm ad}([[x_{1},x_{2}],x_{3}]))\\
&=D([[{\rm ad}x_{1},{\rm ad}x_{2}],{\rm ad}x_{3}])\\
 &=[[D({\rm ad}x_{1}),{\rm ad}x_{2}],{\rm ad}x_{3}]+(-1)^{\left|D\right|\left|x_{1}\right|}[[{\rm ad}x_{1},D({\rm ad}x_{2})],{\rm ad}x_{3}]\\
 &+(-1)^{\left|D\right|(\left|x_{1}\right|+\left|x_{2}\right|)}[[{\rm ad}x_{1},{\rm ad}x_{2}],D({\rm ad}x_{3})].\\
 &=[[{\rm ad}(d(x_{1})),{\rm ad}x_{2}],{\rm ad}x_{3}]+(-1)^{\left|D\right|\left|x_{1}\right|}[[{\rm ad}x_{1},{\rm ad}(d(x_{2}))],{\rm ad}x_{3}]\\
 &+(-1)^{\left|D\right|(\left|x_{1}\right|+\left|x_{2}\right|)}[[{\rm ad}x_{1},{\rm ad}x_{2}],{\rm ad}(d(x_{3}))]\\
 &={\rm ad}([[d(x_{1}),x_{2}],x_{3}]+(-1)^{\left|D\right|\left|x_{1}\right|}[[x_{1},d(x_{2})],x_{3}]+(-1)^{\left|D\right|(\left|x_{1}\right|+\left|x_{2}\right|)}[[x_{1},x_{2}],d(x_{3})]).
\end{align*}
Since ${\rm Z}(L)=0$,
\begin{align*}
 &d([[x_{1},x_{2}],x_{3}])=[[d(x_{1}),x_{2}],x_{3}]+(-1)^{\left|D\right|\left|x_{1}\right|}[[x_{1},d(x_{2})],x_{3}]+(-1)^{\left|D\right|(\left|x_{1}\right|+\left|x_{2}\right|)}[[x_{1},x_{2}],d(x_{3})].
\end{align*}
That is to say, $d\in {\rm TDer}(L)$. By $\rm Lemma~2.7$, $d\in {\rm
Der}(L)$.
\end{proof}
{\bf Proof of $\rm\bf Theorem~1.1$}. By $\rm Lemma~2.7$, it remains
only to prove the second assertion. By $\rm Lemma~2.10$, for
$\forall D\in {\rm TDer}({\rm Der}(L)),~x\in L$, there exists $d\in
{\rm Der}(L)$ such that for $\forall x\in L,~D({\rm ad}x)= {\rm
ad}(d(x))$. Thanks to $\rm Lemma~2.6$, ${\rm ad}(d(x))=[d,{\rm
ad}x]$. Hence,
\begin{align*}
 &D({\rm ad}x)= {\rm ad}(d(x))=[d,{\rm ad}x]={\rm ad}(d)({\rm ad}x).
\end{align*}
Thus,
\begin{align*}
 &(D-{\rm ad}(d))({\rm ad}x)=0.
\end{align*}
By $\rm Lemma~2.9$, $D={\rm ad}(d)$. Therefore, ${\rm TDer}({\rm
Der}(L))={\rm ad}({\rm Der}(L))$. The theorem holds.

\begin{re}{\rm\cite{Z1}}
The condition $\frac{1}{2}$ is necessary. For example, if the base
ring is field $F$ of characteristic $2$ and $L$ is not abelian, then
the identity map is a triple derivation but not a derivation.
\end{re}

\section{Triple Homomorphisms of Perfect Lie Superalgebras}

Throughout this section, $L$ and $L^{'}$ are Lie
superalgebras over $R$, $f$ is a triple homomorphism from $L$ to
$L^{'}$, and $M$ is the enveloping Lie superalgebra of $f(L)$. We
always assume that $L$ is perfect and that $M$ is centerless and can
be decomposed into a direct sum of indecomposable ideals. We proceed
to prove the theorem by a series of lemmas.

\begin{lem}\label{lem:1}
There exists an even $R$-linear mapping $\delta_{f}: L \rightarrow
L^{'}$ such that for $\forall x\in hg(L)$ with $x=\sum_{i\in I\atop
\left|x_{i1}\right|+\left|x_{i2}\right|=\left|x\right|}[x_{i1},x_{i2}]~(x_{i1},x_{i2}\in
L),~\delta_{f}(x)=\sum_{i\in I\atop
\left|x_{i1}\right|+\left|x_{i2}\right|=\left|x\right|}[f(x_{i1}),f(x_{i2})]$.
\end{lem}
\begin{proof}
It is sufficient to prove that $\sum[f(x_{i1}),f(x_{i2})]$ is
independent of the expression of $x$. Suppose that
$x=\sum[x_{i1},x_{i2}]=\sum[y_{j1},y_{j2}].$

Let
\begin{equation*}
\begin{split}
&\alpha=\sum[f(x_{i1}),f(x_{i2})],~\beta=\sum[f(y_{j1}),f(y_{j2})].
\end{split}
\end{equation*}
For $\forall z\in L$, we have that
\begin{equation*}
\begin{split}
&[f(z),\alpha-\beta]=[f(z),\sum[f(x_{i1}),f(x_{i2})]-\sum[f(y_{j1}),f(y_{j2})]]\\
&=\sum[f(z),[f(x_{i1}),f(x_{i2})]]-\sum[f(z),[f(y_{j1}),f(y_{j2})]]\\
&=f([z,\sum[x_{i1},x_{i2}]]-[z,\sum[y_{j1},y_{j2}]])\\
&=f([z,x]-[z,x])=0.\\
\end{split}
\end{equation*}
It follows $\alpha-\beta\in {\rm Z}(M)$ and thus $\alpha=\beta$
since $M$ is centerless. This completes the proof.
\end{proof}

\begin{lem}\label{lem:1}
Let $\delta_{f}$ be the mapping in $\rm Lemma~3.1$. Then, for
$\forall x\in L$, we have that $f{\rm ad}x={\rm ad}
{\delta_{f}(x)}f.$
\end{lem}
\begin{proof}
Let $x=\sum[x_{i1},x_{i2}]\in L$. For $\forall z\in L$, we have
\begin{align*}
 &f{\rm ad}{x}(z)=f([x,z])=\sum f([[x_{i1},x_{i2}],z])\\
&=\sum[[f(x_{i1}),f(x_{i2})],f(z)]=[\sum[f(x_{i1}),f(x_{i2})],f(z)]\\
 &=[\delta_{f}(x),f(z)]={\rm ad}{\delta_{f}(x)}f(z).
\end{align*}
Thus $f{\rm ad}x={\rm ad}{\delta_{f}(x)}f$, and the lemma holds.
\end{proof}

\begin{lem}\label{lem:1-2}
The mapping $\delta_{f}$ is a homomorphism of Lie superalgebras.
\end{lem}
\begin{proof}
For $\forall x,y\in hg(L),z\in L$, it follows from $\rm Lemmas~3.1$
and $\rm 3.2$ that
\begin{equation*}
\begin{split}
&[\delta_{f}([x,y])-[\delta_{f}(x),\delta_{f}(y)],f(z)]\\
&=[\delta_{f}([x,y]),f(z)]-[\delta_{f}(x),[\delta_{f}(y),f(z)]]+(-1)^{\left|x\right|\left|y\right|}[\delta_{f}(y),[\delta_{f}(x),f(z)]]\\
&=[[f(x),f(y)],f(z)]-[\delta_{f}(x),{\rm ad}{\delta_{f}(y)}f(z)]]+(-1)^{\left|x\right|\left|y\right|}[\delta_{f}(y),{\rm ad}{\delta_{f}(x)}f(z)]\\
&=[[f(x),f(y)],f(z)]-[\delta_{f}(x),f([y,z])]+(-1)^{\left|x\right|\left|y\right|}[\delta_{f}(y),f([x,z])]\\
&=[[f(x),f(y)],f(z)]-{\rm ad}{\delta_{f}(x)}f([y,z])+(-1)^{\left|x\right|\left|y\right|}{\rm ad}{\delta_{f}(y)}f([x,z])\\
&=f([[x,y],z])-f([x,[y,z]])+(-1)^{\left|x\right|\left|y\right|}f([y,[x,z]])\\
&=f([[x,y],z]-[x,[y,z]]+(-1)^{\left|x\right|\left|y\right|}[y,[x,z]])\\
&=0.
\end{split}
\end{equation*}
The last equality is due to the Jacobi identity. Since $M$ is the
enveloping Lie superalgebra of $f(L)$ and $z$ is arbitrary in $L$,
$\delta_{f}([x,y])-[\delta_{f}(x),\delta_{f}(y)]\in {\rm Z}(M)$.
Hence, $\delta_{f}([x,y])=[\delta_{f}(x),\delta_{f}(y)]$ since $M$
is centerless. Therefore, the lemma follows from the arbitrariness
of $x,y\in L$.
\end{proof}

\begin{lem}\label{lem:1-3}
Denote $M^{+}={\rm Im}(f+\delta_{f}),~M^{-}={\rm Im}(f-\delta_{f})$.
Then, $M^{+}$ and $M^{-}$ are both ideals of $M$.
\end{lem}
\begin{proof}
Similar with{\rm\cite{Z2}}.
\end{proof}

\begin{lem}\label{lem:2}
$[M^{+},M^{-}]=0.$
\end{lem}
\begin{proof}
Take $x,y,z\in hg(L)$, we have that

\begin{equation*}
\begin{split}
&[[f(x)+\delta_{f}(x),f(y)-\delta_{f}(y)],f(z)]\\
&=[[f(x),f(y)],f(z)]-[[f(x),\delta_{f}(y)],f(z)]\\
&+[[\delta_{f}(x),f(y)],f(z)]-[[\delta_{f}(x),\delta_{f}(y)],f(z)]\\
&=f([[x,y],z])+(-1)^{\left|x\right|\left|y\right|}[{\rm ad}{\delta_{f}(y)}f(x),f(z)]+[{\rm ad}{\delta_{f}(x)}f(y),f(z)]\\
&-[\delta_{f}(x),[\delta_{f}(y),f(z)]]+(-1)^{\left|x\right|\left|y\right|}[\delta_{f}(y),{\rm ad}{\delta_{f}(x)}f(z)]\\
&=f([[x,y],z])+(-1)^{\left|x\right|\left|y\right|}[f([y,x]),f(z)]+[f([x,y]),f(z)]\\
&-[\delta_{f}(x),{\rm ad}{\delta_{f}(y)}f(z)]+(-1)^{\left|x\right|\left|y\right|}[\delta_{f}(y),f([x,z])]\\
&=f([[x,y],z])+(-1)^{\left|x\right|\left|y\right|}[f([y,x]),f(z)]+[f([x,y]),f(z)]\\
&-[\delta_{f}(x),f([y,z])]-(-1)^{\left|x\right|\left|y\right|}(-1)^{\left|y\right|(\left|x\right|+\left|z\right|)}[f([x,z]),\delta_{f}(y)]\\
&=f([[x,y],z])-[\delta_{f}(x),f([y,z])]-(-1)^{\left|y\right|\left|z\right|}[f([x,z]),\delta_{f}(y)]\\
&=f([[x,y],z])-f([x,[y,z]])+(-1)^{\left|y\right|\left|z\right|}(-1)^{\left|y\right|(\left|x\right|+\left|z\right|)}f([y,[x,z]])\\
&=f([[x,y],z]-[x,[y,z]]+(-1)^{\left|x\right|\left|y\right|}[y,[x,z]])=0.
\end{split}
\end{equation*}
Therefore, $[f(x)+\delta_{f}(x),f(y)-\delta_{f}(y)]\in {\rm Z}(M)$.
Since ${\rm Z}(M)=0$, we have that
$[f(x)+\delta_{f}(x),f(y)-\delta_{f}(y)]=0$. The lemma follows.
\end{proof}

\begin{lem}\label{lem:2}
$M^{+}\cap M^{-}=0.$
\end{lem}
\begin{proof}
Similar with{\rm\cite{Z2}}.
\end{proof}

\begin{lem}\label{lem:2-2}
If $M$ can not be decomposed into a direct sum of two nontrivial
ideals. Then $f$ is either a homomorphism or an anti-homomorphism of
Lie superalgebras.
\end{lem}
\begin{proof}
For $\forall x\in L$, let
$m^{+}=\frac{1}{2}(f(x)+\delta_{f}(x)),~m^{-}=\frac{1}{2}(f(x)-\delta_{f}(x))$.
Then $m^{+}\in M^{+},~m^{-}\in M^{-}$ and $f(x)=m^{+}+m^{-}.$ Hence,
$f(L)\subseteq M^{+}+M^{-}.$ Therefore, $M\subseteq M^{+}+M^{-}$. By
$\rm Lemma~3.6$, $M=M^{+}\oplus M^{-}.$ Since $M$ cannot be
decomposed into direct sum of two nontrivial ideals, either $M^{+}$
or $M^{-}$ must be trivial. If $M^{+}$ is trivial i.e.
$(f+\delta_{f})([x,y])=0$, then
$f([x,y])=-\delta_{f}([x,y])=-[f(x),f(y)]=(-1)^{\left|x\right|\left|y\right|}[f(y),f(x)].$
So $f$ is an anti-homomorphism. If $M^{-}$ is trivial i.e.
$(f-\delta_{f})([x,y])=0$, then
$f([x,y])=\delta_{f}([x,y])=[f(x),f(y)].$ So $f$ is a homomorphism.
Hence the lemma follows.
\end{proof}

{\bf Proof of Theorem 1.2}. By $\rm Lemma~3.7$, it remains to prove
the theorem in case $M$ is decomposable. By the assumptions, $M$ can
be written as the sum$$M=M_{1}\oplus M_{2}\oplus...\oplus M_{s},$$
where each $M_{i}$ is an indecomposable ideal of $M$. Since $M$ is
centerless, each $M_{i}$ is also centerless (See $\rm Lemma~3.1$ in
{\rm\cite{C}}). Let $p_{i}$ be the projection of $M$ into $M_{i}$.
Then, $f=\sum_{i=1}^{s}p_{i}f$ and $p_{i}f$ is a triple
homomorphisms from $L$ to $M_{i}$, and $M_{i}$ is the enveloping Lie
superalgebras of $p_{i}f(L)$ for $i=1,2...s.$ Since each $M_{i}$ is
indecomposable, by $\rm Lemma~3.7$,
 $p_{i}f$ is either a homomorphism or an anti-homomorphism from $L$
to $M_{i}$. Let $P=\{i|~p_{i}f~is~a~homomorphism\},$ $Q$ is the
complementary set of $P$ in the set $\{1,2...,s\}$. Let
$M_{1}=\sum_{i\in P}M_{i}$, $M_{2}=\sum_{i\in Q}M_{i}$. Let
$f_{1}=\sum_{i\in P}p_{i}f,$ $f_{1}=\sum_{i\in Q}p_{i}f$. It can be
checked by direct verification that $M=M_{1}\oplus
M_{2},~[M_{1},M_{2}]=0.~f=f_{1}+f_{2},$ $f_{1}$ is a homomorphism
and $f_{2}$ is an anti-homomorphism of Lie superalgebras. The
theorem is proved.

\noindent {\bf Acknowledgements}\quad The authors would like to
thank the referee for valuable comments and suggestions on this
article.


\begin{thebibliography}{99}

\bibitem{B} Bresar, M. (1993). Commuting traces of biadditive mappings, commutativity-preserving
mappings and Lie mappings. Trans. Amer. Math. Soc. 335:525-546.
\bibitem{BC} Beidar, K. I., Chebotar, M. A. (2001). On Lie derivations of Lie ideals of Prime
algebras. Israel J. Math. 123:131-148.
\bibitem{C} Chun, J. H., Lee, J. S. (1996). On Complete Lie
Superalgebras. Comm. Korean Math. Soc. 11:323-334.
\bibitem{HJ} Herstein, I. N., Jordan, L. (1961). Structure in simple, associative rings. Bulletin of the
Amer. Math. Soc. 67:517-531.
\bibitem{K}  Kac V. G.. (1977). Lie Superalgebras. Advances in mathematics 26:8-96.
\bibitem{M1} Miers, C. R. (1971). Lie homomorphisms of operator algebras. Pacific J. Math. 38:
717-737.
\bibitem{M2} Miers, C. R. (1976). Lie triple homomorphisms into von Neumann algebras. Proc.
Amer. Math. Soc. 58:169-172.
\bibitem{M3} Miers, C. R. (1978). Lie triple derivations of Von Neumann algebras. Proc. Amer. Math.
Soc. 71:57-61.
\bibitem{M4} M{\rm$\ddot u$}ller, D. (1989). Isometries of bi-invariant pseudo-Riemannian metries on Lie groups. Geom. Deicata 29:65-96.
\bibitem{J} Jacobson, N., Rickart, C. E. (1950). Jordan homomorphisms of rings. Trans. Amer.
Math. Soc. 69:479-502.
\bibitem{JW} Ji, P. S., Wang, L. (2005). Lie triple derivations of TUHF algebras. Linear Algebra
Applications 403:399-408.
\bibitem{S} Swain, G. A. (1996). Lie derivations of the skew elements of prime ring with involution.
J. Algebra 184:679-704.
\bibitem{ZHC} Zhang, J. H., Wu, B. W., Cao, H. X. (2006). Lie triple derivations of nest algebras.
Linear Algebra Applications 416:559-567.
\bibitem{Z1} Zhou, J. H.. (2013). Triple Derivations of Perfect Lie
Algebras. Comm. Algebra 41:1647-1654.
\bibitem{Z2} Zhou, J. H.. (2014). Triple Homomorphisms of Perfect Lie Algebras. Comm. Algebra 42:3724-3730.
\end{thebibliography}
\end{document}